\documentclass[11pt,reqno,sumlimits]{amsart}
\usepackage{amsfonts, amsmath, amscd, amssymb, euscript, amsthm, array, booktabs,
dcolumn, shortvrb, tabularx, units, url, mathrsfs}
\usepackage[pdftex]{graphicx}
\usepackage[pdftex]{hyperref}
\usepackage[all]{xy}
\normalsize

\DeclareMathOperator{\todd}{Todd}

\DeclareMathOperator{\an}{an}
\DeclareMathOperator{\ind}{ind}

\DeclareMathOperator{\spin}{spin}

\DeclareMathOperator{\odd}{odd}
\DeclareMathOperator{\even}{even}
\DeclareMathOperator{\im}{Im}

\DeclareMathOperator{\ch}{ch}
\DeclareMathOperator{\GL}{GL}
\DeclareMathOperator{\U}{U}

\DeclareMathOperator{\SO}{SO}

\DeclareMathOperator{\str}{Str}

\begin{document}
\setlength{\baselineskip}{1.4\baselineskip}
\theoremstyle{definition}
\newtheorem{defi}{Definition}
\newtheorem{remark}{Remark}
\newtheorem{coro}{Corollary}
\newtheorem{exam}{Example}
\newtheorem{thm}{Theorem}
\newtheorem{prop}{Proposition}
\newcommand{\wt}[1]{{\widetilde{#1}}}
\newcommand{\ov}[1]{{\overline{#1}}}
\newcommand{\wh}[1]{{\widehat{#1}}}
\newcommand{\poin}{Poincar$\acute{\textrm{e }}$}
\newcommand{\deff}[1]{{\bf\emph{#1}}}
\newcommand{\boo}[1]{\boldsymbol{#1}}
\newcommand{\abs}[1]{\lvert#1\rvert}
\newcommand{\norm}[1]{\lVert#1\rVert}
\newcommand{\inner}[1]{\langle#1\rangle}
\newcommand{\poisson}[1]{\{#1\}}
\newcommand{\biginner}[1]{\Big\langle#1\Big\rangle}
\newcommand{\set}[1]{\{#1\}}
\newcommand{\Bigset}[1]{\Big\{#1\Big\}}
\newcommand{\BBigset}[1]{\bigg\{#1\bigg\}}
\newcommand{\dis}[1]{$\displaystyle#1$}
\newcommand{\R}{\mathbb{R}}
\newcommand{\N}{\mathbb{N}}
\newcommand{\Z}{\mathbb{Z}}
\newcommand{\Q}{\mathbb{Q}}
\newcommand{\E}{\mathcal{E}}
\newcommand{\G}{\mathcal{G}}
\newcommand{\F}{\mathcal{F}}
\newcommand{\V}{\mathcal{V}}
\newcommand{\W}{\mathcal{W}}
\newcommand{\SSS}{\mathcal{S}}
\newcommand{\h}{\mathbb{H}}
\newcommand{\g}{\mathfrak{g}}
\newcommand{\C}{\mathbb{C}}
\newcommand{\A}{\mathcal{A}}
\newcommand{\M}{\mathcal{M}}
\newcommand{\HH}{\mathcal{H}}
\newcommand{\D}{\mathcal{D}}
\newcommand{\PP}{\mathcal{P}}
\newcommand{\K}{\mathcal{K}}
\newcommand{\RR}{\mathcal{R}}
\newcommand{\RRR}{\mathscr{R}}
\newcommand{\DDD}{\mathscr{D}}
\newcommand{\so}{\mathfrak{so}}
\newcommand{\gl}{\mathfrak{gl}}
\newcommand{\aaa}{\mathbb{A}}
\newcommand{\bbb}{\mathbb{B}}
\newcommand{\DD}{\mathsf{D}}
\newcommand{\ccc}{\bold{c}}
\newcommand{\sss}{\mathbb{S}}
\newcommand{\cdd}[1]{\[\begin{CD}#1\end{CD}\]}
\normalsize
\title[On an index theorem by Bismut]{On an index theorem by Bismut}
\author{Man-Ho Ho}
\address{Department of Mathematics\\ Hong Kong Baptist University}
\email{homanho@hkbu.edu.hk}
\maketitle
\nocite{*}
\begin{abstract}
In this paper we give a proof of an index theorem by Bismut. As a consequence we obtain
another proof of the Grothendieck--Riemann--Roch theorem in differential cohomology.
\end{abstract}
\tableofcontents
\section{Introduction}
The differential Grothendieck--Riemann--Roch theorem \cite[Theorem 6.19]{BS09},
\cite[Corollary 8.26]{FL10}, \cite[Theorem 1]{H14} (abbreviated as dGRR) is a lift of
the classical Grothendieck--Riemann--Roch theorem to differential cohomology. It states
that for a proper submersion $\pi:X\to B$ with closed $\spin^c$ fibers of even relative
dimension, the following diagram commutes.
\begin{equation}\label{eq 1}
\begin{CD}
\wh{K}(X) @>\wh{\ch}>> \wh{H}^{\even}(X; \R/\Q) \\ @V\ind^{\an}VV @VV\wh{\int_{X/B}}
\wh{\todd}(\wh{\nabla}^{T^VX})\ast(\cdot) V \\ \wh{K}(B) @>>\wh{\ch}> \wh{H}^{\even}
(B; \R/\Q)
\end{CD}
\end{equation}
Here $\wh{K}$ is differential $K$-theory \cite{BS10a, SS08a, FL10} and $\wh{H}$ is
Cheeger--Simons differential characters \cite{CS85, BB14}.

In \cite{H14} the proof of the dGRR is to reduce it to an index theorem by Bismut
\cite[Theorem 1.15]{B05}:
\begin{equation}\label{eq 2}
\wh{\ch}(\ker(\DD^E), \nabla^{\ker(\DD^E)})+i_2(\wt{\eta}(\E))=\wh{\int_{X/B}}
\wh{\todd}(T^VX, \wh{\nabla}^{T^VX})\ast\wh{\ch}(E, \nabla^E).
\end{equation}
One can regard (\ref{eq 2}) as a lift of the local family index theorem \cite{BC89}
to differential characters. Bismut's proof of (\ref{eq 2}) involves certain adiabatic
limits arguments given in \cite{BC89, D91, L94} and an Atiyah--Patodi--Singer index
theorem in differential characters \cite[Theorem 9.2]{CS85}. In this paper we give a
proof of (\ref{eq 2}), which is inspired by \cite{BB14} and does not make use of the
above results.

Section 2 contains the background material needed in this paper. Section 3 contains
the main result of this paper.

\section*{Acknowledgement}
The author would like to thank Thomas Schick for pointing out a mistake in an earlier
version of the paper.

\section{Background Materials}
\subsection{Cheeger--Simons differential characters}

We recall some basic properties of differential characters, and refer to 
\cite{CS85, BB14} for the details.

Let $X$ be a smooth manifold and $k\geq 1$, and $A$ a proper subring of $\R$. A degree 
$k$ differential character $f$ with coefficients in $\R/A$ is a group homomorphism 
$f:Z_{k-1}(X)\to\R/A$ with a fixed $\omega_f\in\Omega^k(X)$ such that for all $c\in 
C_k(X)$, \dis{f(\partial c)=\int_c\omega_f\mod A}. The abelian group of degree $k$ 
differential characters is denoted by $\wh{H}^k(X; \R/A)$. Denote by $\Omega^k_A(X)$ 
the group of closed $k$-forms with periods in $A$. It is easy to see that $\omega_f\in
\Omega^k_A(X)$ and is uniquely determined by $f\in\wh{H}^k(X; \R/A)$. The map $\delta_1:
\wh{H}^k(X; \R/A)\to\Omega^k_A(X)$ by $\delta_1(f)=\omega_f$. The map 
\dis{i_2:\frac{\Omega^{k-1}(X)}{\Omega^{k-1}_A(X)}\to\wh{H}^k(X; \R/A)}, defined by 
\dis{i_2(\alpha)(z)=\int_z\alpha\mod A}, is injective. In the following diagram, every 
square and triangle commutes, and the two diagonal sequences are exact.
\begin{equation}\label{eq 3}
\xymatrix{\scriptstyle 0 \ar[dr] & \scriptstyle & \scriptstyle & \scriptstyle & \scriptstyle 0
\\ & \scriptstyle H^{k-1}(X; \R/A) \ar[rr]^{-B} \ar[dr]^{i_1} & \scriptstyle & \scriptstyle
H^k(X; A) \ar[ur] \ar[dr]^r & \scriptstyle \\ \scriptstyle H^{k-1}(X; \R) \ar[ur]^{\alpha}
\ar[dr]_{\beta} & \scriptstyle & \scriptstyle\wh{H}^k(X; \R/A) \ar[ur]^{\delta_2}
\ar[dr]^{\delta_1} & \scriptstyle & \scriptstyle H^k(X; \R) \\ \scriptstyle & \scriptstyle
\frac{\Omega^{k-1}(X)}{\Omega^{k-1}_{A}(X)} \ar[rr]_{d} \ar[ur]^{i_2} & \scriptstyle &
\scriptstyle\Omega^k_{A}(X) \ar[ur]^s \ar[dr] & \scriptstyle \\ \scriptstyle 0 \ar[ur] &
\scriptstyle & \scriptstyle & \scriptstyle & \scriptstyle 0}
\end{equation}
The maps $\delta_1$ and $\delta_2$ are called the curvature and the characteristic class
in literatures respectively.

There is a unique ring structure for $\wh{H}^*(X; \R/A)$ \cite[Corollary 32]{BB14},
denoted by $\ast$. For a fiber bundle $\pi:X\to B$ with closed oriented fibers, the
``integration along the fiber", denoted by \dis{\wh{\int_{X/B}}}, exists and is unique
\cite[Theorem 39]{BB14}.

\subsection{Index bundle and Bismut--Cheeger eta form}

In this subsection we recall the construction of the index bundle and of the
Bismut--Cheeger eta form. We refer to \cite{BGV, LM89} for the details.

Let $E\to X$ be a complex vector bundle with a Hermitian metric $h$ and $\nabla$ a
unitary connection on $E\to X$. Let $\pi:X\to B$ be a proper submersion of even relative
dimension $n$, and $T^VX\to X$ the vertical tangent bundle which is assumed to have a
metric $g^{T^VX}$. A given horizontal distribution $T^HX\to X$ and a Riemannian metric
$g^{TB}$ on $B$ determine a metric on $TX\to X$  by $g^{TX}:=g^{T^VX}\oplus\pi^*g^{TB}$.
If $\nabla^{TX}$ is the corresponding Levi-Civita connection on $TX\to X$, then
$\nabla^{T^VX}:=P\circ\nabla^{TX}\circ P$ is a connection on $T^VX\to X$, where $P:TX\to
T^VX$ is the orthogonal projection. Assume the vertical bundle $T^VX\to X$ has a
$\spin^c$-structure; i.e., the principal $\SO(n)$-bundle $\SO(T^VX)\to X$ admits a
$\spin^c(n)$ reduction. Denote by $S^c(T^VX)\to X$ the spinor bundle, which is a 
complex vector bundle, associated to the $\spin^c$ structure of $T^VX\to X$. Note that 
$S^c(T^VX)\cong S(T^VX)\otimes L^{\frac{1}{2}}(X)$, where $S(T^VX)$ is the spinor bundle
associated to the locally defined $\spin$ structure of $T^VX\to X$, and $L^{\frac{1}{2}}
(X)$ is a locally defined Hermitian line bundle (see \cite[(D.19)]{LM89}) Note that 
their tensor product is globally defined. Write $\nabla^{T^VX}$ for both the Levi-Civita
connection on $T^VX\to X$ and also its lift to $S(T^VX)$. Choose a unitary connection 
$\nabla^{L^VX}$ on $L^{\frac{1}{2}}(X)$. Define a connection $\wh{\nabla}^{T^VX}$ on 
$S^c(T^VX)\to X$ by $\wh{\nabla}^{T^VX}:=\nabla^{T^VX}\otimes\nabla^{L^VX}$. The Todd 
form $\todd(\wh{\nabla}^{T^VX})$ of $S^c(T^VX)\to X$ is defined by
$$\todd(\wh{\nabla}^{T^VX}):=\wh{A}(\nabla^{T^VX})\wedge e^{\frac{1}{2}c_1(\nabla^{L^VX})}.$$

The Bismut--Cheeger eta form \dis{\wt{\eta}(\E)\in\frac{\Omega^{\odd}(B)}{\im(d)}}
associated to $\E:=(E, h, \nabla)$ is defined as follows. Consider the infinite-rank
superbundle $\pi_*E\to B$, where the fibers at each $b\in B$ is given by
$$(\pi_*E)_b:=\Gamma(X_b, (S^c(T^VX)\otimes E)|_{X_b}).$$
Recall that $\pi_*E\to B$ admits an induced Hermitian metric and a connection
$\nabla^{\pi_*E}$ compatible with the metric \cite[\S 9.2, Proposition 9.13]{BGV}. For
each $b\in B$, the canonically constructed Dirac operator
$$\DD^E_b:\Gamma(X_b, (S^c(T^VX)\otimes E)|_{X_b})\to\Gamma(X_b, (S^c(T^VX)\otimes 
E)|_{X_b})$$
gives a family of Dirac operators, denoted by $\DD^E:\Gamma(X, S^c(T^VX)\otimes E)\to
\Gamma(X, S^c(T^VX)\otimes E)$. Assume the family of kernels $\ker(\DD^E_b)$ has locally 
constant dimension, i.e., $\ker(\DD^E)\to B$ is a finite-rank Hermitian superbundle. Let
$P:\pi_*E\to\ker(\DD^E)$ be the orthogonal projection, $h^{\ker(\DD^E)}$ be the
Hermitian metric on $\ker(\DD^E)\to B$ induced by $P$, and $\nabla^{\ker(\DD^E)}:=P\circ
\nabla^{\pi_*E}\circ P$ be the connection on $\ker(\DD^E)\to B$ compatible to
$h^{\ker(\DD^E)}$.

The scaled Bismut superconnection $\aaa_t:\Omega(B, \pi_*E)\to\Omega(B, \pi_*E)$
\cite[Definition 3.2]{B86} (see also \cite[Proposition 10.15]{BGV} and \cite[(1.4)]{D91}),
is defined by
$$\aaa_t:=\sqrt{t}\DD^E+\nabla^{\pi_*E}-\frac{c(T)}{4\sqrt{t}},$$
where $c(T)$ is the Clifford multiplication by the curvature 2-form of the fiber bundle.
The Bismut--Cheeger eta form $\wt{\eta}(\E)$ \cite[(2.26)]{BC89} (see also \cite{D91}
and \cite[Theorem 10.32]{BGV}) is defined by
$$\wt{\eta}(\E):=\frac{1}{2\sqrt{\pi}}\int^\infty_0\frac{1}{\sqrt{t}}\str\bigg
(\frac{d\aaa_t}{dt}e^{-\aaa_t^2}\bigg)dt.$$
The local family index theorem states that
\begin{equation}\label{eq 4}
d\wt{\eta}(\E)=\int_{X/B}\todd(\wh{\nabla}^{T^VX})\wedge\ch(\nabla^E)-\ch(\nabla^{\ker
(\DD^E)}).
\end{equation}

Let $f:\wt{B}\to B$ be a smooth map. In \cite[\S 2.3.2]{BS09} the pullback of the above
geometric data is studied, and, in particular, the Bunke eta form (\cite[Definition 2.2.16]
{B09}) is shown to respect pullback. Since the Bismut--Cheeger eta form is a special case
of the Bunke eta form, we have
\begin{equation}\label{eq 5}
\wt{\eta}(f^*\E)=f^*\wt{\eta}(\E).
\end{equation}
One can also prove (\ref{eq 5}) directly as in \cite[\S 2.3.2]{BS09}.

\section{Main result}

In this section we prove the main result in this paper. We employ the setup and 
assumptions made in Section 2.2. For write $\E$ for $(E, h, \nabla)$, where $E\to X$ is 
a Hermitian bundle with a Hermitian metric $h$ and $\nabla$ a unitary connection on 
$E\to X$.

As in \cite{H14} it suffices to prove (\ref{eq 2}) in the special case where the family
of kernels of the Dirac operators has constant dimension, i.e., $\ker(\DD^E)\to B$ is a
superbundle. The general case of (\ref{eq 2}) follows from a standard perturbation
argument as in \cite[\S 7]{FL10} and its proof is essentially the same as the special
case.

\begin{prop}\label{prop 1}
For any $\E$, the differential character
\begin{equation}\label{eq 6}
\wh{\int_{X/B}}\wh{\todd}(T^VX, \wh{\nabla}^{T^VX})\ast\wh{\ch}(E, \nabla)-\wh{\ch}
(\ker(\DD^E), \nabla^{\ker(\DD^E)})
\end{equation}
is uniquely characterized by the conditions that it is natural and its curvature is 
given by
\begin{equation}\label{eq 7}
\int_{X/B}\todd(\wh{\nabla}^{T^VX})\wedge\ch(\nabla)-\ch(\nabla^{\ker(\DD^E)}).
\end{equation}
\end{prop}
\begin{proof}
The differential character in (\ref{eq 6}) obviously satisfies the conditions. The
uniqueness of (\ref{eq 6}) come from \cite[Proposition 3.1]{H15}. However, 
\cite[Proposition 3.1]{H15} is only correct under a restrictive class of Lie 
groups\footnote{The author would like to thank Thomas Schick for bringing up this point.} 
$G$, namely, it is valid if the cohomology of $BG$ is torsion free. Example of such $G$ 
includes the stable general linear group $\GL(\C)$, and therefore the stable unitary 
group $\U$. Thus \cite[Proposition 3.1]{H15} is true for all differential characteristic 
classes of complex vector bundles of even degree, and therefore it can still be applied 
to our case.
\end{proof}
Bismut's theorem follows from the observation that the differential character 
$i_2(\wt{\eta}(\E))\in\wh{H}^{\even}(B; \R/\Q)$ satisfies the conditions stated in 
Proposition \ref{prop 1}: The naturality of $i_2(\wt{\eta}(\E))$ follows from (\ref{eq 5}). 
The curvature of $i_2(\wt{\eta}(\E))$ is given by (\ref{eq 7}) is a consequence of the 
commutativity of the lower triangle of (\ref{eq 3}) and the local family index theorem 
(\ref{eq 4}).
\bibliographystyle{amsplain}
\bibliography{MBib}
\end{document}